\documentclass{amsart}

\newtheorem{theorem}{Theorem}[section]
\newtheorem{lemma}[theorem]{Lemma}

\theoremstyle{definition}

\newtheorem{proposition}[theorem]{Proposition}
\newtheorem{corollary}[theorem]{Corollary}
\theoremstyle{remark}
\newtheorem{remark}[theorem]{Remark}

\numberwithin{equation}{section}

\begin{document}

\title{Harmonic operators on convolution quantum group algebras}

%    Remove any unused author tags.

%    author one information
\author[M. Nemati]{Mehdi Nemati$^{1,2}$}
\address{}
\curraddr{}
\email{}
\thanks{}
\email{m.nemati@iut.ac.ir}
\email{simasoltani@iut.ac.ir}
%    author two information
\author[S. Soltani Renani]{Sima Soltani Renani$^{1}$}
\address{}
\curraddr{}
\email{}
\thanks{}
\address{$^1$ Department of Mathematical Sciences, Isfahan University of Technology,
	Isfahan 84156-83111, Iran}
\address{$^2$ School of Mathematics,
	Institute for Research in Fundamental Sciences (IPM),
	P.O. Box: 19395--5746, Tehran, Iran}
%\subjclass[2020]{Primary  22D15, 43A07, Secondary   43A22, 46H05.}

%\keywords{}

%\date{}

\dedicatory{}

\begin{abstract}
	Let ${\Bbb G}$ be a locally compact quantum group and
	${\mathcal T}(L^2({\Bbb G}))$ be the Banach algebra of trace class operators on $L^2({\Bbb G})$ with the convolution induced by the right fundamental unitary of ${\Bbb G}$. We study the space of harmonic operators $\widetilde{\mathcal H}_\omega$ in ${\mathcal B}(L^2({\Bbb G}))$ associated to a contractive element $\omega\in {\mathcal T}(L^2({\Bbb G}))$. We characterize the existence of non-zero harmonic operators   in ${\mathcal K}(L^2({\Bbb G}))$ and relate them with some properties of the quantum group ${\Bbb G}$, such as finiteness, amenability and co-amenability.
\end{abstract}

\maketitle
\let\thefootnote\relax

\section*{Introduction}
Let $\mu$ be  a complex  Borel measure on  a locally compact group $G$. A function $f\in L^\infty(G)$ 
is called {$\mu$-harmonic} if   it satisfies the convolution equation $\mu*f=f$.  
This concept, especially when $\mu$
is a probability measure and its support generates
$G$, has been extensively studied in the context of random walks; see for example \cite{az, kv}.
Of interest here is the collection of all $\mu$-harmonic functions, denoted by ${\mathcal H}_\mu$. This is an
abelian von Neumann algebra, but only for a twisted product, and is only a subalgebra
of $L^\infty(G)$ when it is trivial. In \cite{cl}, the dual analogue of this concept
was studied; that is, replacing $L^\infty(G)$ by the group von Neumann algebra $VN(G)$ and
the measure algebra $M(G)$ by the Fourier-Stieltjes algebra $B(G)$.

On the other hand, in \cite{jn} the authors investigated the concept of a general $\mu$-harmonic
operator on $L^2(G)$. The key tool for this was introduced by Ghahramani \cite{ghah}, who defined an isometric representation
$\theta$ of $M(G)$ on ${\mathcal B}(L^2(G))$ which extends the convolution action of  $M(G)$ on $L^\infty(G)$, where $L^\infty(G)$ acts on $L^2(G)$ by multiplication. Hence we can define an operator $T$ to be $\mu$-harmonic if $\theta(\mu)(T)=T$. 
%The collection of all $\mu$-harmonic operators is denoted by $\widetilde{\mathcal H}_\mu$, and $\widetilde{\mathcal H}_\mu$ is a crossed product of ${\mathcal H}_\mu$  with $G$.

The dual version of the map $\theta$ has been constructed in  \cite{nrs}. Here the completely bounded multipliers
of the Fourier algebra $A(G)$, the algebra $M_{cb}(A(G))$, which contains $B(G)$, is represented in
a completely isometric way on ${\mathcal B}(L^2(G))$ by a map $\widehat{\theta}$. Subsequently,  for $\sigma\in M_{cb}(A(G))$  the authors of \cite{nr} define an operator $T\in {\mathcal B}(L^2(G))$ to be $\sigma$-harmonic if $\widehat{\theta}(\sigma)(T)=T$ and  studied the $\sigma$-harmonic operators in ${\mathcal B}(L^2(G))$.

Motivated
by these observations, harmonic operators in the setting of locally compact quantum groups
have been studied in \cite{knr, nss}. In particular they investigated the structure of  $\mu$-harmonic operators, so in other words  `fixed
point spaces', denoted by ${\mathcal H}_\mu$, associated to arbitrary quantum contractive measures $\mu$. One of the main results
of \cite{knr} is that when $\mu$ is a  quantum probability
measure on ${\Bbb G}$ then the space ${\mathcal H}_\mu$ is a von
Neumann algebra, but with a product usually different from the one in $L^\infty({\Bbb G})$. In particular, if $\mu$ is non-degenerate, then ${\mathcal H}_\mu$  is a subalgebra of $L^\infty({\Bbb G})$  if and only if ${\mathcal H}_\mu={\Bbb C}1$. 
%In \cite{knr},  for a locally compact quantum group ${\Bbb G}$, the authors introduced  and studied
%the notion of $\mu$-harmonic operators ${\mathcal H}_\mu$, where  $\mu$ is a universal
%quantum probability measures on ${\Bbb G}$.
% state in  $M_u({\Bbb G})$, the Poisson
%boundary of G, denoted by H, is dened as the xed point algebra

Finally, in \cite{knr2}  for a locally compact quantum group ${\Bbb G}$ 
with a convolution action by a quantum probability measure, the abstract structure of noncommutative
harmonic operators on the
level of ${\mathcal B}(L^2({\Bbb G}))$ is studied in detail and connected to the crossed products of von Neumann algebras.

%In what follows, ${\Bbb G}$  denotes a locally  compact quantum group in the setting .......... 

It is
known that the right fundamental unitary of a locally compact quantum group  ${\Bbb G}$ induces a completely contractive multiplication
$\triangleright$ on the space ${\mathcal T}(L^2({\Bbb G}))$ of trace class operators on $L^2({\Bbb G})$; see \cite{hnr2011,  hnr}. The multiplication $\triangleright$ on ${\mathcal T}(L^2({\Bbb G}))$ induces a natural ${\mathcal T}(L^2({\Bbb G}))$-bimodule structure on ${\mathcal B}(L^2({\Bbb G}))$. We consider the convolution
algebra $\mathcal{T}_\triangleright({\mathbb G}):=\left(\mathcal{T}\left(L^{2}(\mathbb{G})\right), \triangleright\right)$, with focus on the left $\mathcal{T}_\triangleright({\mathbb G})$-module action
on ${\mathcal B}(L^2({\Bbb G}))$ and study harmonic operators on the
level of ${\mathcal B}(L^2({\Bbb G}))$. The paper is organized as follows.

%2-As one of the main results in this paper, we  will try to extend the main results in \cite{chu3} to an arbitrary locally compact hypergroup $K$.
In Section  \ref{sec2}, 
some preliminary definitions and results on locally compact
quantum groups and Banach algebras which are needed, are briefly recalled.

In Section  \ref{sec3}, 
for given a contractive element  $\omega\in \mathcal{T}_\triangleright({\mathbb G})$, we first show that there is a contractive projection from ${\mathcal B}(L^2({\Bbb G}))$  onto $\widetilde{\mathcal H}_\omega$ of $\omega$-harmonic operators; that is,  fixed points in ${\mathcal B}(L^2({\Bbb G}))$ under the action of $\omega$.
%$=\{x\in {\mathcal B}(L^2({\Bbb G})): L_\omega(x)= \omega\triangleright x=x\}$.
For the case that $\omega$ is a non-degenerate state in $\mathcal{T}_\triangleright({\mathbb G})$, we show that $\widetilde{\mathcal H}_\omega$ can be equipped
with a product, different from the one in ${\mathcal B}(L^2({\Bbb G}))$, turning it into a
von Neumann algebra, but $\widetilde{\mathcal H}_\omega$  is a subalgebra of ${\mathcal B}(L^2({\Bbb G}))$  if and only if $\widetilde{\mathcal H}_\omega=L^\infty(\widehat{\Bbb G})$. In this case we prove that ${\Bbb G}$ is compact if and only if $\widetilde{\mathcal H}_\omega\cap {\mathcal K}(L^2({\Bbb G}))\neq\{0\}$.

In Section  \ref{sec4},
we turn our attention to the  pre-annihilator $J_\omega$
of $\widetilde{\mathcal H}_\omega$ as a left ideal in $\mathcal{T}_\triangleright({\mathbb G})$ and use it to characterize some properties of  ${\Bbb G}$ such as finiteness, amenability and co-amenability.
For example, we show that  ${\Bbb G}$ is necessarily amenable if there exists a  state $\omega\in \mathcal{T}_\triangleright({\mathbb G})$  such that $\widetilde{\mathcal H}_\omega=L^\infty(\widehat{\Bbb G})$. We prove that the converse is  also true when $\mathcal{T}_\triangleright({\mathbb G})$ is separable.

In Section \ref{sec5}, 
for the case that ${\Bbb G}$ is discrete and $\omega$ is a contractive element in $\mathcal{T}_\triangleright({\mathbb G})$ we show that the Cesaro sums $\omega_n=\frac{1}{n}\sum_{k=1}^n \omega^k$ do not converge to zero in the 
weak$^*$ topology of  $\mathcal{T}_\triangleright({\mathbb G})$ if and only if $\widetilde{\mathcal H}_\omega\cap {\mathcal K}(L^2({\Bbb G}))\neq\{0\}$.\\

{\section { Preliminaries}}\label{sec2}
The class of 
locally compact quantum groups was first introduced and studied
by Kustermans and Vaes \cite{ku-va1,ku-va2}. Recall that a 
({\it von Neumann algebraic}) {\it locally compact
	quantum group} is a
quadruple ${\mathbb G}=(L^\infty({\mathbb G}), \Gamma, \phi, \psi)$, 
where $L^\infty({\mathbb G})$ is a von Neumann algebra with identity element $1$ and a co-multiplication 
$
\Gamma : L^\infty({\mathbb G})\rightarrow L^\infty({\mathbb G})\bar{\otimes}L^\infty({\mathbb G}).
$
Moreover,  $\phi$ and $\psi$ are normal faithful semifinite left and right Haar weights 
on $L^\infty({\mathbb G})$, respectively.
Here $\bar{\otimes}$ denotes the von Neumann algebra tensor product. 

%The Gelfand-Naimark-Segal construction applied to the left Haar
%weight $\phi$ of  every locally compact quantum group $\mathbb{G}$ gives a
% Hilbert space $L^2({\mathbb G})$. There exists a left fundamental unitary operator
%$W$ on $L^2({\mathbb G})\otimes L^2({\mathbb G})$  implementing the co-multiplication $\Gamma$ via
%$$\Gamma(x)=W^*(1\otimes x)W~~~(x\in L^\infty({\mathbb G})).$$
%For more
%details see \cite{ku-va1, ku-va2}.

The predual of $L^\infty({\mathbb G})$ is denoted by $L^1({\mathbb G})$ which is called
\textit{quantum group algebra} of 
$\mathbb{G}$ and the Hilbert space associated with $\phi$ or $\psi$
is denoted by $L^2({\Bbb G})$.
Then $L^\infty({\Bbb G})$ is standardly represented on $L^2({\Bbb G})$ and  the
pre-adjoint of the co-multiplication $\Gamma$ induces on $L^1({\mathbb G})$
an associative completely contractive multiplication
$\Gamma_*:L^1({\mathbb G})\widehat{\otimes}L^1({\mathbb G})\rightarrow L^1({\mathbb G})$,
where $\widehat{\otimes}$ is the operator space projective tensor product.
Therefore, $L^1({\mathbb G})$
is a Banach algebra under the product $\star$ given by
$$f\star g:=\Gamma_*(f\otimes g)\in L^1({\mathbb G})$$ for all $f,g\in L^1({\mathbb G})$.
%$\langle f\star g, x\rangle=\langle f\otimes g, \Gamma x\rangle$ for all $f,g\in L^1({\mathbb G})$ and $x\in L^\infty({\mathbb G})$.
Moreover,  the module actions of  $L^1({\mathbb G})$ on $L^\infty({\mathbb G})$ are given by
$$f\star x:=(\iota\otimes f)(\Gamma(x)),\quad x\star f:=(f\otimes\iota)(\Gamma(x))$$
for all $f\in L^1({\mathbb G})$ and $x\in L^\infty({\mathbb G})$.

For every locally compact quantum group ${\Bbb G}$, there exists  a
left fundamental unitary operator $W$ on $L^2({\Bbb G})\otimes L^2({\Bbb G})$ and  a right fundamental
unitary operator $V$ on $L^2({\Bbb G})\otimes L^2({\Bbb G})$ such that the co-multiplication $\Gamma$ on $L^\infty({\Bbb G})$ can be expressed as
$$
\Gamma(x)=W^*(1\otimes x)W=V(x\otimes 1)V^*\quad(x\in L^\infty({\Bbb G})).
$$
The left regular representation $\lambda: L^1({\Bbb G})\rightarrow {\mathcal B}(L^2({\Bbb G}))$ is defined by
$$
\lambda(f)=(f\otimes \iota)(W)\quad (f\in L^1({\Bbb G})),
$$
which is an injective and completely contractive algebra homomorphism from $L^1({\Bbb G})$ into
${\mathcal B}(L^2({\Bbb G}))$. Then $L^\infty(\widehat{\mathbb G})={\{\lambda(f): f\in L^1(\mathbb G)\}}^{''}$ is the von Neumann algebra associated with the
dual quantum group $\widehat{\mathbb G}$ of ${\mathbb G}$. 
Similarly, we have the right regular representation
$\rho: L^1({\Bbb G})\rightarrow {\mathcal B}(L^2({\Bbb G}))$ defined by
$$
\rho(f)=(\iota\otimes f)(V)\quad (f\in L^1({\Bbb G})),
$$
which is also an injective and completely contractive algebra homomorphism from $L^1({\Bbb G})$ into
${\mathcal B}(L^2({\Bbb G}))$. Then $L^\infty(\widehat{\mathbb G}')={\{\rho(f): f\in L^1(\mathbb G)\}}^{''}$ is the von Neumann algebra associated with the
quantum group $\widehat{\mathbb G}'$. Moreover, we have  $W\in L^\infty(\mathbb{G})\bar{\otimes} L^\infty(\widehat{\mathbb{G}})$, $V\in L^\infty(\widehat{\mathbb{G}}')\bar{\otimes} L^\infty(\mathbb{G})$ and $L^\infty(\widehat{\mathbb G})'=L^\infty(\widehat{\mathbb G}')$.

The \textit{reduced quantum group $C^*$-algebra} of $L^\infty({\mathbb G})$ is defined as
$$C_0({\mathbb G}):=\overline{\{(\iota\otimes\omega)(W);\ \omega\in {\mathcal B}(L^2({\mathbb G}))_*\}}^{\|.\|}.$$
%We say that ${\mathbb G}$ is \textit{compact} if $C_0({\mathbb G})$ is a unital $C^*$-algebra.
We say that ${\mathbb G}$ is   compact if  $1\in C_0({\mathbb G})$ 
and is discrete if the dual quantum group $\widehat{\mathbb G}$
of ${\mathbb G}$ is compact, which is equivalent
to $L^1({\mathbb G})$ being unital; see \cite{rund2}.

The co-multiplication $\Gamma$ maps $C_0({\mathbb G})$ into the
multiplier algebra $M(C_0({\mathbb G})\otimes C_0({\mathbb G}))$ of the minimal $C^*$-algebra tensor product $C_0({\mathbb G})\otimes C_0({\mathbb G})$. Thus, we can define the completely contractive
product $\star$ on $C_0({\mathbb G})^*=M({\mathbb G})$ by
$$
\langle \omega\star\nu, x\rangle=(\omega\otimes\nu)(\Gamma x)\quad (x\in C_0({\mathbb G}), \omega,\nu\in M({\mathbb G}))
$$
whence  $(M({\mathbb G}), \star)$ is a completely contractive Banach algebra  and contains $L^1({\mathbb G})$ as a norm
closed two-sided ideal. We recall that a
left invariant mean on $L^\infty({\mathbb G})$, is a state $m\in L^\infty({\mathbb G})^*$ satisfying 
$$
\langle m, x\star f\rangle=\langle f, 1\rangle \langle m, x\rangle\quad (f\in L^1({\mathbb G}), x\in L^\infty({\Bbb G})).
$$
Right and (two-sided) invariant means are defined similarly.
A locally compact quantum
group ${\mathbb G}$ is said to be amenable if there exists a left (equivalently, right or two-sided) invariant mean on $L^\infty({\mathbb G})$; see \cite[Propodition 3]{dqvaes}. We also recall that, $\mathbb{G}$ is called \emph{co-amenable} if $L^1({\mathbb G})$ has a bounded approximate identity. The subspace $LUC({\mathbb G})$ of $L^\infty({\mathbb G})$ is defined by $LUC({\mathbb G})=\langle L^\infty({\mathbb G})\star L^1({\mathbb G})\rangle$, where $\langle\cdot\rangle$ denotes the closed linear span. Moreover, we have the following inclusions
$$
C_0({\Bbb G})\subseteq LUC({\mathbb G})\subseteq M(C_0({\Bbb G})).
$$ 
The right fundamental unitary $V$ of $\mathbb{G}$ induces a co-associative co-multiplication
$$
\widetilde{\Gamma}: \mathcal{\mathcal B}\left(L^{2}(\mathbb{G})\right) \ni x \mapsto V(x \otimes 1) V^{*} \in {\mathcal B}(L^{2}(\mathbb{G})) \bar{\otimes} L^\infty(\mathbb{G}),
$$
and $\widetilde{\Gamma}|_{L^{\infty}(\mathbb{G})}=\Gamma$. The pre-adjoint of $\widetilde{\Gamma}$ induces an associative completely contractive multiplication on space $\mathcal{T}(L^{2}(\mathbb{G}))={\mathcal B}(L^{2}(\mathbb{G}))_*$  of trace class operators on $L^{2}(\mathbb{G})$, defined by
$$
\triangleright: \mathcal{T}(L^{2}(\mathbb{G})) \widehat{\otimes} \mathcal{T}(L^{2}(\mathbb{G})) \ni \omega \otimes \tau \mapsto \omega \triangleright \tau=\widetilde{\Gamma}_{*}(\omega \otimes \tau) \in \mathcal{T}(L^{2}(\mathbb{G})).
$$
It was shown in \cite[Lemma 5.2]{hnr2011}, that the pre-annihilator
$L^{\infty}(\mathbb{G})_{\perp}$ of $L^{\infty}(\mathbb{G})$ in $\mathcal{T}(L^{2}(\mathbb{G}))$ 
is a norm closed two-sided ideal in
$(\mathcal{T}(L^{2}(\mathbb{G})), \triangleright)$ 
and the complete quotient map
$$
\pi: \mathcal{T}(L^{2}(\mathbb{G})) \ni \omega \mapsto f=\left.\omega\right|_{L^{\infty}(\mathbb{G})} \in L^{1}(\mathbb{G})
$$
is a completely contractive algebra homomorphism from 
$\mathcal{T}_\triangleright({\mathbb G}):=\left(\mathcal{T}\left(L^{2}(\mathbb{G})\right), \triangleright\right)$
onto $L^{1}(\mathbb{G})$. We always have $\langle \mathcal{T}_\triangleright({\mathbb G})\triangleright\mathcal{T}_\triangleright({\mathbb G})\rangle=\mathcal{T}_\triangleright({\mathbb G})$ and the multiplication $\triangleright$ defines  a canonical
$\mathcal{T}_\triangleright({\mathbb G})$-bimodule structure on
${\mathcal B}\left(L^{2}(\mathbb{G})\right)$. It is also known from \cite[Proposition 5.3]{hnr2011} that $\langle {\mathcal B}(L^2({\mathbb G}))\triangleright\mathcal{T}_\triangleright({\mathbb G})\rangle=LUC({\mathbb G})$.  In particular,   the  actions of $\mathcal{T}_\triangleright({\mathbb G})$ on $L^{\infty}(\mathbb{G})$ satisfies
$$\omega \triangleright x=\pi(\omega)\star x, \quad x\triangleright \omega=x\star\pi(\omega)$$
for all $\omega\in \mathcal{T}_\triangleright({\mathbb G})$ and $ x\in L^{\infty}(\mathbb{G})$. Let ${\mathcal K}(L^2({\mathbb G}))$ be the $C^*$-algebra of compact operators on $L^2({\mathbb G})$. Then the  
equality $\langle {\mathcal K}(L^2({\mathbb G}))\triangleright\mathcal{T}_\triangleright({\mathbb G})\rangle= C_0({\mathbb G})$ was established in \cite{hnr}.

Note that since $V \in L^{\infty}(\widehat{\mathbb{G}}^{\prime}) \bar{\otimes} L^{\infty}(\mathbb{G})$,
the bimodule action  of $\mathcal{T}_\triangleright({\mathbb G})$ on $L^{\infty}(\widehat{\mathbb{G}})$ becomes rather trivial. In fact, for
$\hat{x} \in L^{\infty}(\widehat{\mathbb{G}})$ and $\omega \in \mathcal{T}_\triangleright({\mathbb G})$
we have
$$\hat{x} \triangleright \omega=(\omega \otimes \iota) V(\hat{x} \otimes 1) V^{*}=\langle\omega, \hat{x}\rangle 1,
\quad \omega \triangleright \hat{x}=(\iota \otimes \omega) V(\hat{x} \otimes 1) V^{*}=\langle\omega, 1\rangle \hat{x}.$$

We recall  some notation related to Banach algebras ${\mathcal A}$. As is well known, ${\mathcal A}^*$ is canonically a
Banach ${\mathcal A}$-bimodule with the actions
$$
\langle x\cdot a, b\rangle=\langle x, ab\rangle, \quad \langle a\cdot x, b\rangle=\langle x, ba\rangle   
$$
for all $a,b\in {\mathcal A}$ and $x\in{\mathcal A}^*$.  It is known that there are two Banach algebra multiplications, $\square$ and $\diamondsuit$  on
$\mathcal{A}^{* *}$, each extending the multiplication  on $\mathcal{A}$. 
For $m, n \in \mathcal{A}^{* *}$ and $x \in \mathcal{A}^{*},$ the left Arens product $\square$ on $\mathcal{A}^{* *}$ is given  by the left $\mathcal{A}$-module structure on $\mathcal{A}$ as follows
$$
\langle m \square n, x\rangle=\langle m, n \square x\rangle
$$
where $n \square x \in \mathcal{A}^{*}$ is defined by $\langle n \square x, a\rangle=\langle n, x \cdot a\rangle$ 
for all $a \in \mathcal{A}$. Similarly, the right Arens product $\Diamond$  is defined by considering ${\mathcal A}$ as a right ${\mathcal A}$-module.\\

\section{{Harmonic operators in ${\mathcal B}(L^2({\Bbb G}))$}}\label{sec3}

Let $m\in{\mathcal B}(L^{2}(\mathbb{G}))^*$.
Then, we can define the right $\mathcal{T}_\triangleright({\mathbb G})$-module
map $L_m$ on ${\mathcal B}(L^{2}(\mathbb{G}))$ via
$$
L_m(x)=m\square  x\quad (x\in {\mathcal B}(L^{2}(\mathbb{G}))),
$$
where $m \square x \in {\mathcal B}(L^{2}(\mathbb{G}))$ is defined by $\langle m \square x, \omega\rangle=\langle m, x \triangleright \omega\rangle$. We also have $\|L_m\|_{cb}\leq\|m\|$
, and if we let $\mathcal{CB}_{\mathcal{T}_\triangleright}({\mathcal B}(L^{2}(\mathbb{G})))$ denote the
algebra of completely bounded right $\mathcal{T}_\triangleright({\mathbb G})$-module maps on $B(L^{2}(\mathbb{G}))$, it
follows that the map  $$\Phi: {\mathcal B}(L^{2}(\mathbb{G}))^*\longrightarrow \mathcal{CB}_{\mathcal{T}_\triangleright}({\mathcal B}(L^{2}(\mathbb{G}))),\quad m\mapsto L_m$$ is a weak$^*$-weak$^*$ continuous, contractive, 
algebra homomorphism. For $m\in{\mathcal B}\left(L^{2}(\mathbb{G})\right)^*$  define $\widetilde{\mathcal H}_m$ to be the set of all $m$-harmonic operators; that is, 
$$
\widetilde{\mathcal H}_m=\{x\in {\mathcal B}(L^2({\Bbb G})): L_m(x)=x\}.
$$
Since for every
$\hat{x} \in L^{\infty}(\widehat{\mathbb{G}})$ and $\omega \in \mathcal{T}_\triangleright({\mathbb G})$
we have $ \hat{x}\triangleright \omega =\langle\omega, \hat{x}\rangle 1$, it follows that for $m\in{\mathcal B}(L^{2}(\mathbb{G}))^*$  with $\langle m, 1\rangle=1$ we obtain $L^\infty(\widehat{\Bbb G})\subseteq\widetilde{\mathcal H}_m$.

Given $\omega\in \mathcal{T}_\triangleright({\Bbb G})$ with $\|\omega\|=1$ and $n\in{\Bbb N}$ we define
$$
\omega_n:=\frac{1}{n}\sum_{k=1}^n \omega^k
$$
where $\omega^k$ stands for the $k$th power of $\omega$ with respect to the product $\triangleright$
in $\mathcal{T}_\triangleright({\Bbb G})$. Now, for a
free ultrafilter $\mathcal{U}$ on ${\Bbb N}$, we consider $m_{\mathcal{U}}$ in ${\mathcal B}(L^2({\Bbb G}))^*$ given by $$m_{\mathcal{U}}:=w^*-\lim_{\mathcal U}\omega_n.$$
We denote by $S(\mathcal{T}_\triangleright({\Bbb G}))$ the set of all states in $\mathcal{T}_\triangleright({\Bbb G})$. We are particularly
interested in the case when  $\omega\in S(\mathcal{T}_\triangleright({\Bbb G}))$. 
\begin{remark}\label{th2}
	Note that, under the natural embedding $\mathcal{T}_\triangleright({\Bbb G})\hookrightarrow \mathcal{T}_\triangleright({\Bbb G})^{**}={\mathcal B}(L^2({\Bbb G}))^*$, we have
	$\omega\square m=\omega\triangleright m$ and $m\square \omega=m\triangleright \omega$ 
	for all $\omega\in\mathcal{T}_\triangleright({\Bbb G})$ and $ m\in\mathcal{T}_\triangleright({\Bbb G})^{**}$, where $\omega\triangleright m$ and $m\triangleright \omega$  are the canonical $\mathcal{T}_\triangleright({\Bbb G})$-module actions on $\mathcal{T}_\triangleright({\Bbb G})^{**}$.
	Now, let $\omega\in \mathcal{T}_\triangleright({\Bbb G})$ with $\|\omega\|=1$ and let $\mathcal{U}$ be a free ultrafilter on ${\Bbb N}$. Then, since $\omega\triangleright\omega_n=\omega_n\triangleright\omega=\omega_n+\frac{1}{n}(\omega^{n+1}-\omega^n)$, it is easily verified that
	$$m_{\mathcal{U}}=m_{\mathcal{U}}\triangleright\omega=\omega\triangleright m_{\mathcal{U}}.$$This shows that $m_{\mathcal{U}}=w^*-\lim_{\mathcal U}(\omega_n\triangleright m_{\mathcal{U}})=m_{\mathcal{U}}\square m_{\mathcal{U}}$, which implies that $m_{\mathcal{U}}$ is  either $0$ or a contractive idempotent. 
	Moreover, it is easy to see that $m_{\mathcal{U}}$ is a state if  $\omega$ is a state.
\end{remark}

Let $\pi: {\mathcal T}_\triangleright({\Bbb G})\rightarrow L^1({\Bbb G})$ be the  quotient map. Then the map 
$$
\pi^{**}: {\mathcal T}_\triangleright({\Bbb G})^{**}\rightarrow L^1({\Bbb G})^{**}
$$ 
is a surjective algebra homomorphism,  where ${\mathcal T}_\triangleright({\Bbb G})^{**}$ and $L^1({\Bbb G})^{**}$ are equipped with their left Arens products. 
\begin{lemma}\label{th00}
	Let $m\in{\mathcal B}\left(L^{2}(\mathbb{G})\right)^*$. Then $L_m|_{L^\infty({\Bbb G})}=L_{\pi^{**}(m)}$, where $L_{\pi^{**}(m)}: L^\infty({\Bbb G})\rightarrow L^\infty({\Bbb G})$ is defined by $L_{\pi^{**}(m)}(x)=\pi^{**}(m)\square x$ for all $x\in L^\infty({\Bbb G})$.
\end{lemma}
\begin{proof}
	Let $m\in{\mathcal B}\left(L^{2}(\mathbb{G})\right)^*$  and  $x\in L^\infty({\Bbb G})$. Take a net $(\omega_i)$ in ${\mathcal T}_\triangleright({\Bbb G})$ such that 
	$\omega_i\rightarrow m$ in the weak$^*$-topology of ${\mathcal B}\left(L^{2}(\mathbb{G})\right)^*$. 
	Then $\omega_i \triangleright x\rightarrow m\square x$ 
	in the weak$^*$-topology of ${\mathcal B}(L^{2}(\mathbb{G}))$. On the other hand, $\omega_i \triangleright x=\pi(\omega_i)\star x\in L^\infty({\Bbb G})$ and $\pi(\omega_i)\star x\rightarrow \pi^{**}(m)\square x$ in the weak$^*$-topology of $L^\infty({\Bbb G})$. Since $L^\infty({\Bbb G})$
	is weak$^*$-closed in ${\mathcal B}(L^{2}(\mathbb{G}))$ we obtain that $m\square x=\pi^{**}(m)\square x$. This shows that
	$L_m|_{L^\infty({\Bbb G})}=L_{\pi^{**}(m)}$.
\end{proof}

The following result is an immediate consequence of the above lemma.
\begin{corollary}\label{cor00}
	Let $m\in{\mathcal B}\left(L^{2}(\mathbb{G})\right)^*$. Then $\widetilde{\mathcal H}_m\cap L^\infty({\Bbb G})=\mathcal{H}_{\pi^{**}(m)}$, where $\mathcal{H}_{\pi^{**}(m)}=\{x\in L^\infty({\Bbb G}): L_{\pi^{**}(m)}(x)=x\}$
\end{corollary}
\begin{lemma}\label{th0}
	Let $\omega\in \mathcal{T}_\triangleright({\Bbb G})$ with $\|\omega\|=1$ and let $\mathcal{U}$ be a free ultrafilter on ${\Bbb N}$. Then the map $L_{m_{\mathcal{U}}}$  is a contractive projection from ${\mathcal B}(L^2({\Bbb G}))$ onto $\widetilde{\mathcal H}_{m_{\mathcal{U}}}=\widetilde{\mathcal H}_\omega$.	
\end{lemma}
\begin{proof}
	It suffice to show that $\widetilde{\mathcal H}_{m_{\mathcal{U}}}=\widetilde{\mathcal H}_\omega$. To prove this, let $x\in {\mathcal B}(L^2({\Bbb G}))$. Then
	$ L_{m_{\mathcal{U}}}(x)=w^*-\lim_{\mathcal U} L_{\omega_n}(x)$ in the weak$^*$ topology of ${\mathcal B}(L^2({\Bbb G}))$. This shows that $\widetilde{\mathcal H}_\omega\subseteq\widetilde{\mathcal H}_{m_{\mathcal{U}}}$. To prove the converse inclusion, given $x\in \widetilde{\mathcal H}_{m_{\mathcal{U}}}$,  we have 
	$$
	L_\omega(x)=L_\omega(L_{m_{\mathcal{U}}}(x))=L_{\omega\triangleright m_{\mathcal{U}}}(x)=L_{m_{\mathcal{U}}}(x)=x,
	$$
	and this completes the proof.
\end{proof}

\begin{theorem}
	Let $\omega\in \mathcal{T}_\triangleright({\Bbb G})$ with $\|\omega\|=1$.  Then the following statements are equivalent.
	
	{\rm (i)} $\widetilde{\mathcal H}_\omega=\{0\}$.
	
	{\rm (ii)} $\widetilde{\mathcal H}_\omega\cap LUC({\Bbb G})=\{0\}$.
	
	{\rm (iii)} $\omega_n\rightarrow 0$ weak$^*$ in $LUC({\Bbb G})^*$.
	
	{\rm (iv)} $m_{\mathcal{U}}|_{LUC({\Bbb G})}=0$ for all free ultra filter $\mathcal{U}$.
	
	{\rm (v)} $m_{\mathcal{U}}|_{LUC({\Bbb G})}=0$ for some free ultra filter $\mathcal{U}$.	
\end{theorem}
\begin{proof}	
	(ii)$\Rightarrow$(i). Let $x\in\widetilde{\mathcal H}_\omega$ be non-zero. By definition, $\omega \triangleright x=x$. If $x\triangleright\gamma=0$ for all $\gamma\in \mathcal{T}_\triangleright({\mathbb G})$, then
	$$
	\langle x, \gamma\rangle=\langle \omega \triangleright x, \gamma\rangle=\langle x\triangleright \gamma, \omega \rangle=0.
	$$
	So, $x=0$ contradiction. Thus, there is some $\gamma\in \mathcal{T}_\triangleright({\mathbb G})$  such that $y:=x\triangleright \gamma \neq 0$. Moreover, it is clear that $y\in \widetilde{\mathcal H}_\omega\cap LUC({\Bbb G})$.
	
	The implications	(i)$\Rightarrow$(ii) and  (iii)$\Leftrightarrow$(iv)$\Rightarrow$(v) are trivial.
	
	(i)$\Rightarrow$(iv). Suppose that $m_{\mathcal{U}}|_{LUC({\Bbb G})}\neq 0$ for some free ultra filter $\mathcal{U}$. Then since  $\langle {\mathcal B}(L^2({\mathbb G}))\triangleright\mathcal{T}_\triangleright({\mathbb G})\rangle=LUC({\mathbb G})$, there is $x\in {\mathcal B}(L^2({\Bbb G}))$ such that 
	$L_{m_{\mathcal{U}}}(x)\neq 0$. On the other hand, $L_\omega(L_{m_{\mathcal{U}}}(x))=L_{\omega\triangleright m_{\mathcal{U}}}(x)=L_{m_{\mathcal{U}}}(x)\neq 0$. This shows that $L_{m_{\mathcal{U}}}(x)\in \widetilde{\mathcal H}_\omega$.
	
	(v)$\Rightarrow$(ii). Let $x\in\widetilde{\mathcal H}_\omega\cap LUC({\Bbb G})$ be non-zero. Then $\omega_n\triangleright x=x$ for all $n\in {\Bbb N}$. Now, given  $\gamma\in \mathcal{T}_\triangleright({\mathbb G})$ with $\langle x, \gamma\rangle \neq 0$, we have
	$$
	\langle m_{\mathcal{U}}, x\triangleright\gamma\rangle=w^*-\lim_{\mathcal U}\langle \omega_n, x\triangleright\gamma\rangle=w^*-\lim_{\mathcal U}\langle \omega_n\triangleright x, \gamma\rangle=\langle x, \gamma\rangle\neq 0.
	$$
	This shows that $m_{\mathcal{U}}|_{LUC({\Bbb G})}\neq 0$.
\end{proof}

%\begin{lemma}
%	Let $\omega\in S(\mathcal{T}_\triangleright({\Bbb G}))$. Then $L_\omega$ is a faithful Markov operator $($i.e., a normal, completely positive and unital map$)$ on ${\mathcal B}(L^2({\Bbb G}))$.
%{lemma}
%\begin{proof}
%By \cite[Lemma 3.4]{knr},
%$L_{\pi(\omega)}$ is faithful on $L^\infty({\Bbb G})$, so that $\iota\otimes L_{\pi(\omega)}$ is faithful on  ${\mathcal B}(L^2(\mathbb{G})) \bar{\otimes} L^\infty({\Bbb G})$. Since $\widetilde{\Gamma}\circ L_\omega=(\iota\otimes L_{\pi(\omega)})\circ \widetilde{\Gamma}$, it follows that $L_\omega$ is faithful on ${\mathcal B}(L^2({\Bbb G}))$. 
%\end{proof}

Let $\omega\in S(\mathcal{T}_\triangleright({\Bbb G}))$. Then the operator $L_\omega$ is a Markov operator, i.e., a unital normal completely positive map, on ${\mathcal B}(L^2({\Bbb G}))$. Although $\widetilde{\mathcal H}_\omega$ is not an algebra in general, it is easy to see that it is a
weak$^*$-closed operator system (i.e. a unital and self-adjoint closed subspace)  in ${\mathcal B}(L^2({\Bbb G}))$. However, we can introduce a new product in $\widetilde{\mathcal H}_\omega$ so that it
becomes a von Neumann algebra. Let us recall this construction
for the convenience of the reader; see \cite[Sec. 2.5]{izu}.

We fix  a free ultrafilter ${\mathcal U}$ on ${\Bbb N}$. Then $L_{m_{\mathcal{U}}}$ is a projection of norm $1$ from ${\mathcal B}(L^2({\Bbb G}))$ onto $\widetilde{\mathcal H}_\omega$ and
the Choi-Effros product $x\bullet y:=L_{m_{\mathcal{U}}}(xy)$ defines a von Neumann algebra product on  $\widetilde{\mathcal H}_\omega$, different, of course, from the one in ${\mathcal B}(L^2(\mathbb{G}))$. 
%It follows from Lemma \ref{th00} that the restriction of $L_{m_{\mathcal{U}}}$  to $L^\infty({\Bbb G})$ is equal
%to  $L_{\pi^{**}(m_{\mathcal{U}})}$. 
%It is easy to see that $\pi(\omega)\in S(L^1({\Bbb G}))$,  $\mathcal{H}_{\pi(\omega)}\subseteq \widetilde{\mathcal H}_\omega$ and with this product $\mathcal{H}_{\pi(\omega)}$ is a von Neumann subalgebra of $\widetilde{\mathcal H}_\omega$. 
Let us stress that the von Neumann algebra  structure
of $\widetilde{\mathcal H}_\omega$ does not depend on the choice of the free ultrafilter ${\mathcal U}$ since every completely positive isometric linear isomorphism between two von Neumann algebras is a $*$-isomorphism.

%We can show, exactly as in the proof of  \cite[Theorem 2.4]{knr2},  that there is a von Neumann algebra isomorphism $\widetilde{\Gamma}_\omega: \widetilde{\mathcal H}_\omega\rightarrow {\mathcal B}(L^2(\mathbb{G})) \bar{\otimes} \mathcal{H}_{\pi(\omega)}$ between $\widetilde{\mathcal H}_\omega$ and the crossed product ${\Bbb G}\ltimes_{\Gamma_{\pi(\omega)}}\mathcal{H}_{\pi(\omega)}=\{\Gamma_{\pi(\omega)}(\mathcal{H}_{\pi(\omega)})\cup(L^\infty(\widehat{\Bbb G})\otimes 1)\}''$, where $\Gamma_{\pi(\omega)}$ is the restriction of $\Gamma$ to $\mathcal{H}_{\pi(\omega)}$ which induces a left coaction  of ${\Bbb G}$ on the von Neumann algebra $\mathcal{H}_{\pi(\omega)}$; see \cite[Proposition 2.1]{knr2}.

\begin{lemma}\label{lem2}
	Let $x\in{\mathcal B}(L^{2}(\mathbb{G}))$ and $\widetilde{\Gamma}(x)\in{\mathcal B}(L^{2}(\mathbb{G})) {\otimes}1$.  Then $x\in L^\infty(\widehat{\mathbb{G}})$.
\end{lemma}
\begin{proof}
	Suppose that	
	$\widetilde{\Gamma}(x)=y\otimes 1$ for some $y\in{\mathcal B}(L^{2}(\mathbb{G}))$. Then $\widetilde{\Gamma}(x)=V(x\otimes 1)V^*=y\otimes 1$, and so  $V(x\otimes 1)=(y\otimes 1)V$. Using  the slice map $(\iota\otimes f)$ to both sides of this equation, we obtain  $\rho(f)x=y\rho(f)$ for all $f\in L^1({\Bbb G})$. Therefore, we have  
	$x=y\in L^\infty(\widehat{\mathbb{G}})$ since $\rho(L^1({\Bbb G}))$ is weak$^*$-dense in $L^\infty(\widehat{\mathbb{G}}')$ and $L^\infty(\widehat{\mathbb{G}})=L^\infty(\widehat{\mathbb{G}}')'$.	
\end{proof}

\begin{theorem}\label{on}
	Let $\omega\in \mathcal{T}_\triangleright({\Bbb G})$ be a state. Then the following statements are equivalent.
	
	{\rm (i)}   ${\mathcal H}_{\pi(\omega)}={\Bbb C}1$.
	
	{\rm (ii)} $\widetilde{\mathcal H}_\omega=L^\infty(\widehat{\Bbb G})$.
\end{theorem}
\begin{proof}
	(i)$\Rightarrow$(ii). Suppose that ${\mathcal H}_{\pi(\omega)}={\Bbb C}1$.  Clearly, $L^\infty(\widehat{\Bbb G})\subseteq \widetilde{\mathcal H}_\omega$. To prove the converse inclusion, let us first suppose that $x\in \widetilde{\mathcal H}_\omega$. Then it is easy to see that $x\triangleright \gamma\in \widetilde{\mathcal H}_\omega\cap L^\infty({\Bbb G})={\mathcal H}_{\pi(\omega)}={\Bbb C}1$ for all $\gamma\in\mathcal{T}_\triangleright({\Bbb G})$. Thus there is a unique complex number  $C_x^\gamma$ such that $x\triangleright \gamma=C_x^\gamma 1$. On the other hand,
	$$
	\langle x, \gamma\rangle=\langle \omega \triangleright x, \gamma\rangle=\langle x\triangleright \gamma, \omega \rangle=\langle C_x^\gamma 1, \omega\rangle=C_x^\gamma.
	$$
	Therefore, for every $\gamma, \sigma\in \mathcal{T}_\triangleright({\mathbb G})$, we have 
	$$
	\langle \widetilde{\Gamma}(x), \gamma\otimes \sigma\rangle=\langle x, \gamma\triangleright \sigma\rangle=\langle x\triangleright \gamma, \sigma\rangle=\langle \gamma, x\rangle\langle \sigma, 1\rangle=\langle x\otimes 1, \gamma\otimes \sigma\rangle. 
	$$
	This shows that $\widetilde{\Gamma}(x)=x\otimes 1$, and so 
	$x\in L^\infty(\widehat{\mathbb{G}})$ by Lemma \ref{lem2}.  Hence, $\widetilde{\mathcal H}_\omega=L^\infty(\widehat{\Bbb G})$.
	The implication (ii)$\Rightarrow$(i) follows from these facts that ${\mathcal H}_{\pi(\omega)}=\widetilde{\mathcal H}_\omega\cap L^\infty({\Bbb G})$  and $L^\infty(\widehat{\Bbb G})\cap L^\infty({\Bbb G})={\Bbb C}1$.
\end{proof}

We call a state $\omega\in S(\mathcal{T}_\triangleright({\Bbb G}))$ non-degenerate if the state $\pi(\omega)$ in $L^1({\Bbb G})$ is non-degenerate
in the sense of \cite{knr}; that is, for every non-zero operator
$x\in C_0({\Bbb G})^+$ there exists $n\in{\Bbb N}$ such that $\langle\omega^n, x\rangle=\langle\pi(\omega)^n, x\rangle>0$.
We also recall that, a locally compact quantum
group ${\Bbb G}$ is said to be finite if $L^\infty({\Bbb G})$ is finite dimensional, which is equivalent
to  ${\Bbb G}$ being both compact and discrete.
\begin{corollary}
	Let $\omega\in S(\mathcal{T}_\triangleright({\Bbb G}))$ be non-degenerate.  Then  ${\Bbb G}$ is finite if and only if ${\mathcal H}_{\pi(\omega)}\cap {\mathcal K}(L^2({\Bbb G}))\neq\{0\}$.
\end{corollary}
\begin{proof}
	This follows from \cite[Theorem 3.7]{knr} and the fact that ${\Bbb G}$ is finite if and only if $1\in {\mathcal K}(L^2({\Bbb G}))$.
\end{proof}
\begin{corollary}
	Let $\omega\in S(\mathcal{T}_\triangleright({\Bbb G}))$ be non-degenerate. Then the following statements are equivalent.
	
	{\rm (i)}   $\widetilde{\mathcal H}_\omega$ is a subalgebra of ${\mathcal B}(L^2({\Bbb G}))$.
	
	{\rm (ii)} $\widetilde{\mathcal H}_\omega=L^\infty(\widehat{\Bbb G})$.
\end{corollary}
\begin{proof}
	(i)$\Rightarrow$(ii). Suppose that $\widetilde{\mathcal H}_\omega$ is a subalgebra of ${\mathcal B}(L^2({\Bbb G}))$. Then ${\mathcal H}_{\pi(\omega)}=\widetilde{\mathcal H}_\omega\cap L^\infty({\Bbb G})$ is a subalgebra of $L^\infty({\Bbb G})$.
	On the other hand, by \cite[Theorem 3.6]{knr} and  non-degeneracy of $\omega$ we have ${\mathcal H}_{\pi(\omega)}={\Bbb C}1$.  By Theorem \ref{on}, we conclude that  $\widetilde{\mathcal H}_\omega=L^\infty(\widehat{\Bbb G})$. As the implication (ii)$\Rightarrow$(i) is trivial,
	we are done.
\end{proof}
%\begin{lemma}\label{onn}
%	Let $\omega\in S(\mathcal{T}_\triangleright({\Bbb G}))$ be non-degenerate.  Then  ${\Bbb G}$ is compact if and only if $\widetilde{\mathcal H}_\omega\cap C_0({\Bbb G})\neq\{0\}$ and in this case 
%	$\widetilde{\mathcal H}_\omega=L^\infty(\widehat{\Bbb G})$.
%\end{lemma}
%\begin{proof}
%	This follows easily from the equality $\widetilde{\mathcal H}_\omega\cap C_0({\Bbb G})={\mathcal H}_{\pi(\omega)}\cap C_0({\Bbb G})$, \cite[Theorem 3.8 and Theorem 5.3]{knr} and Theorem \ref{on}.
%\end{proof}
\begin{theorem}\label{com}
	Let $\omega\in S(\mathcal{T}_\triangleright({\Bbb G}))$ be non-degenerate.  Then the following statements are equivalent. 
	
	{\rm (i)}  ${\Bbb G}$ is compact.
	
	{\rm (ii)} $\widetilde{\mathcal H}_\omega\cap C_0({\Bbb G})\neq\{0\}$.
	
	{\rm (iii)} $\widetilde{\mathcal H}_\omega\cap {\mathcal K}(L^2({\Bbb G}))\neq\{0\}$.
	
	In all of these cases, $\widetilde{\mathcal H}_\omega=L^\infty(\widehat{\Bbb G})$.
	% $\widetilde{\mathcal H}_\omega\cap {\mathcal K}(L^2({\Bbb G}))\neq\{0\}$ and in this case
\end{theorem}
\begin{proof}
	Th equivalence (i)$\Leftrightarrow$(ii) follows  from the equality $\widetilde{\mathcal H}_\omega\cap C_0({\Bbb G})={\mathcal H}_{\pi(\omega)}\cap C_0({\Bbb G})$ and \cite[Theorem 3.8]{knr}. 
	
	(i)$\Rightarrow$(iii). Suppose that  ${\Bbb G}$ is compact. Then  $\widehat{\Bbb G}$ is discrete and hence it follows by \cite[Theorem 3.7]{hnr}, applied to $\widehat{\Bbb G}$, that $C_0(\widehat{\Bbb G})\subseteq {\mathcal K}(L^2({\Bbb G}))$. Moreover, $L^\infty(\widehat{\Bbb G})\subseteq \widetilde{\mathcal H}_\omega$. These show that $C_0(\widehat{\Bbb G})\subseteq \widetilde{\mathcal H}_\omega\cap {\mathcal K}(L^2({\Bbb G}))$. 
	
	(iii)$\Rightarrow$(i). Suppose that $x\in\widetilde{\mathcal H}_\omega\cap {\mathcal K}(L^2({\Bbb G}))$ is non-zero.
	Since $\widetilde{\mathcal H}_\omega\cap {\mathcal K}(L^2({\Bbb G}))$ is generated by its self-adjoint elements, we can assume that $x$ is self-adjoint and $\|x\|=1$. Without loss of generality, we can find a state $\mu\in \mathcal{T}_\triangleright({\Bbb G})={\mathcal K}(L^2({\Bbb G}))^*$ such that $\langle \mu, x\rangle=\|x\|$. If $x\neq 1$, then   $1-x$ is  a non-zero positive operator in $\widetilde{\mathcal H}_\omega$. Therefore,  $(1-x)\triangleright \mu=1-x\triangleright \mu$ is a positive operator in $\widetilde{\mathcal H}_\omega\cap LUC({\Bbb G})$. Now, suppose that $1-x\triangleright \mu\neq 0$. Then by \cite[Lemma 3.3]{knr} and non-degeneracy of $\omega$, we conclude that there is $n\in{\Bbb N}$ such that 
	$$
	\langle \omega^n, 1-x\triangleright \mu\rangle >0.
	$$
	On the other hand, since $x\in \widetilde{\mathcal H}_\omega$, we have   $\omega^n\triangleright x=x$. Therefore, 
	\begin{eqnarray*}
		\langle \omega^n, 1-x\triangleright \mu\rangle&=&1-\langle \omega^n, x\triangleright \mu\rangle\\
		&=&1-\langle \omega^n\triangleright x, \mu\rangle\\
		&=&1-\langle\mu, x\rangle=0,
	\end{eqnarray*}
	which is a contradiction. Thus, $1=x\triangleright \mu\in C_0({\Bbb G})$, which implies that ${\Bbb G}$ is compact. The last statement follows from  \cite[Theorem 5.3]{knr} and Theorem \ref{on}.
\end{proof}

\section{{Ideals $J_\omega$ and harmonic operators}}\label{sec4}
%\section{Annihilators of $\widetilde{\mathcal H}_\omega$ in $\mathcal{T}_\triangleright({\Bbb G})$}
Let ${\Bbb G}$ be a locally compact quantum  group. Given $\omega\in \mathcal{T}_\triangleright({\Bbb G})$ with $\|\omega\|=1$, the set
$$J_\omega:=\overline{\{\gamma-\gamma\triangleright\omega: \gamma\in \mathcal{T}_\triangleright({\Bbb G})\}}^{\|\cdot\|}$$  is a closed left ideal in the convolution algebra
$\mathcal{T}_\triangleright({\Bbb G})$. Moreover, it is easy to see that the annihilator of  $J_\omega$ in ${\mathcal B}(L^2({\Bbb G}))$ is equal to $\widetilde{\mathcal H}_\omega$. Let $\omega\in S(\mathcal{T}_\triangleright({\mathbb G}))$. Then it will be useful to see that $\gamma-\gamma\triangleright\omega_n$ belongs to $J_\omega$ for all $\gamma\in \mathcal{T}_\triangleright({\mathbb G})$ and $n\in{\Bbb N}$, where $\omega_n$ is the Cesaro sums $\frac{1}{n}\sum_{k=1}^n \omega^k$. Moreover, it is easy to see that $\lim_{n\rightarrow \infty}(\gamma-\gamma \triangleright \omega)\triangleright \omega_n=0$ for all $\gamma\in \mathcal{T}_\triangleright({\mathbb G})$.
Thus, in this case the ideal $J_\omega$ can be expressed as follows
$$
J_\omega=\{\gamma\in \mathcal{T}_\triangleright({\Bbb G}): \lim_{n\rightarrow\infty}\|\gamma\triangleright\omega_n\|=0 \}.
$$
This shows that if $(e_{i})_{{i}\in \Lambda}$ is a bounded right approximate identity for $\mathcal{T}_\triangleright({\mathbb G})$,  then the double-indexed net $(e_{i}-e_{i}\triangleright\omega_n)_{n\in{\Bbb N}, {i}\in \Lambda}$ is a 
bounded right approximate identity for the left ideal $J_\omega$.
\begin{remark}
	We recall that
	the bimodule action  of $\mathcal{T}_\triangleright({\mathbb G})$ on $L^{\infty}(\widehat{\mathbb{G}})$ satisfies
	$$\hat{x} \triangleright \omega=\langle\omega, \hat{x}\rangle 1,
	\quad \omega \triangleright \hat{x}=\langle\omega, 1\rangle \hat{x},$$
	for all $\hat{x} \in L^{\infty}(\widehat{\mathbb{G}})$ and $\omega \in \mathcal{T}_\triangleright({\mathbb G})$.
	This implies that	the pre-annihilator
	$L^\infty(\widehat{\mathbb{G}})_{\perp}:=\{\omega \in \mathcal{T}_\triangleright({\mathbb G}): \omega|_{L^\infty(\widehat{\mathbb{G}})}= 0\}$ of $L^\infty(\widehat{\mathbb{G}})$
	is a two-sided ideal in $\mathcal{T}_\triangleright({\mathbb G})$ which is contained in the augmentation ideal
	$\mathcal{T}_\triangleright({\mathbb G})_0:=\{\omega\in\mathcal{T}_\triangleright({\mathbb G}): \langle \omega, 1\rangle=0\}$
	and for each $\omega\in \mathcal{T}_\triangleright({\mathbb G})$ with $\langle \omega, 1\rangle=1$, $L^\infty(\widehat{\mathbb{G}})_{\perp}$ contains the left ideal $J_\omega$. Finally, we note that the multiplication $\triangleright$ induces a
	multiplication on the quotient algebra $\mathcal{T}_\triangleright({\mathbb G})/L^\infty(\widehat{\mathbb{G}})_{\perp}\cong L^1(\widehat{\mathbb{G}})$, also denoted by $\triangleright$.
	This multiplication, however, is not the usual convolution product on $L^1(\widehat{\mathbb{G}})$. Indeed, it is
	easy to check  that
	$$
	\hat{f}\triangleright \hat{g}=\langle\hat{g}, 1\rangle\hat{f}\quad (\hat{f},\hat{g}\in L^1(\widehat{\mathbb{G}})).
	$$
	Thus, every $\hat{e}\in L^1(\widehat{\mathbb{G}})$ with $\langle \hat{e}, 1\rangle=1$ is a right identity for $\mathcal{T}_\triangleright({\mathbb G})/L^\infty(\widehat{\mathbb{G}})_{\perp}\cong L^1(\widehat{\mathbb{G}})$. 
\end{remark}

\begin{theorem}\label{a-co}
	Let ${\Bbb G}$ be a locally compact quantum group. Then   the following statements hold.
	
	\rm{(i)}  $L^\infty(\widehat{\mathbb{G}})_{\perp}$ has a  bounded right approximate identity if and only if  $\mathbb{G}$  is co-amenable and  amenable.
	
	\rm{(ii)} $L^\infty(\widehat{\mathbb{G}})_{\perp}$ has a right identity if and only if $\mathbb{G}$  is finite.	
\end{theorem}
\begin{proof}
	(i). Suppose that the ideal $L^\infty(\widehat{\mathbb{G}})_{\perp}$ has a bounded right approximate identity. Since the quotient algebra $\mathcal{T}_\triangleright({\mathbb G})/L^\infty(\widehat{\mathbb{G}})_{\perp}$ has also a right identity, we can build a bounded right approximate identity  for $\mathcal{T}_\triangleright({\mathbb G})$; see \cite[Pg. 43]{doran}. Hence, ${\Bbb G}$ is co-amenable by \cite[Proposition 5.4]{hnr2011}. Now, let $E$  be a weak$^*$ cluster point in ${\mathcal B}(L^2({\Bbb G}))^{*}$ of a  bounded right approximate identity in $L^\infty(\widehat{\mathbb{G}})_{\perp}$. Putting $P:={\rm id}-L_E$, it is easy to see that $P$ is a unital right  $\mathcal{T}_\triangleright({\mathbb G})$-module projection from ${\mathcal B}(L^2({\Bbb G}))$ onto $L^\infty(\widehat{\mathbb{G}})$. Since $L_E|_{L^\infty({\Bbb G})}=L_{\pi^{**}(E)}$, it follows that $P(L^\infty({\Bbb G}))\subseteq L^\infty({\Bbb G})\cap L^\infty(\widehat{\Bbb G})={\Bbb C}1$. Thus, by restriction
	there is a unique functional $m\in L^\infty({\Bbb G})^*$ satisfying $P(x)=\langle m, x\rangle 1$ for all $x\in L^\infty({\Bbb G})$. It is easy to see that for every $x\in L^\infty({\Bbb G})$ and  $f\in L^1({\Bbb G})$, we have 
	$$\langle m, x\star f\rangle 1=P(x\triangleright \omega)=P(x)\triangleright\omega=\langle \omega, 1\rangle\langle m , x\rangle 1=\langle f, 1\rangle\langle m , x\rangle 1,$$
	where $\omega\in \mathcal{T}_\triangleright({\mathbb G})$ is a contractive normal extension of $f$.
	Moreover, since   $P(1)=1$, we conclude that $\langle m, 1\rangle=1$. This shows that $m$ is a left invariant functional on $L^\infty({\Bbb G})$, which implies that ${\Bbb G}$ is amenable by \cite[Theorem 2.1]{Ru}. For the converse, first we note that co-amenability of ${\Bbb G}$ implies that $\mathcal{T}_\triangleright({\mathbb G})$ has a bounded right approximate identity by \cite[Proposition 5.4]{hnr2011}. Now, using \cite[Theorem 4.2]{cran} and amenability of ${\Bbb G}$, we may find a norm-one  projection   $P: {\mathcal B}(L^2({\Bbb G}))\rightarrow L^\infty(\widehat{\mathbb{G}})$ in $\mathcal{CB}_{\mathcal{T}_\triangleright}({\mathcal B}\left(L^{2}(\mathbb{G})\right))$. This means that $L^\infty(\widehat{\mathbb{G}})=(L^\infty(\widehat{\mathbb{G}})_\perp)^\perp$ is right invariantly complemented and so $L^\infty(\widehat{\mathbb{G}})_\perp$ has a bounded right approximate identity by \cite[Proposition 6.4]{for}.
	
	(ii). Suppose that  $L^\infty(\widehat{\mathbb{G}})_{\perp}$ has a right identity. Similarly to the first part, one can show that $\mathcal{T}_\triangleright({\mathbb G})$ has a right identity, which implies that ${\Bbb G}$ is discrete by \cite[Proposition 3.7]{kn}. Let $e$  be a right identity for $L^\infty(\widehat{\mathbb{G}})_{\perp}$. Putting $P:={\rm id}-L_e$, it is easy to see that $P$ is  a  unital normal right  $\mathcal{T}_\triangleright({\mathbb G})$-module projection from ${\mathcal B}(L^2({\Bbb G}))$ onto $L^\infty(\widehat{\mathbb{G}})$. Using the normality of $P$ and  a similar argument used in  the proof of part (i), we can show that there is a normal left invariant functional on $L^\infty({\Bbb G})$, which implies that ${\Bbb G}$ is compact by a slight generalization of  \cite[Proposition 3.1]{bet}. Thus, ${\Bbb G}$ is finite. The converse, is trivial.
\end{proof}
%\begin{proposition}
%	Let ${\Bbb G}$ be a locally compact quantum group. Then   the following statements hold.

%	\rm{(i)} If $L^\infty(\widehat{\mathbb{G}})_{\perp}$ has a left identity, then $\mathbb{G}$  is compact.
%	
%	\rm{(ii)} If  $L^\infty(\widehat{\mathbb{G}})_{\perp}$ has a  bounded left approximate identity, %then  $\mathbb{G}$  amenable.	
%\end{proposition}
\begin{proposition}\label{ame}
	Let ${\Bbb G}$ be a locally compact quantum group such that  there exists a state $\omega\in \mathcal{T}_\triangleright({\mathbb G})$ with $\widetilde{\mathcal H}_\omega=L^\infty(\widehat{\mathbb{G}})$. Then $\mathbb{G}$ is amenable.
\end{proposition}
\begin{proof} Let   $\mathcal{U}$ be a free ultrafilter on ${\Bbb N}$ and let $L_{m_{\mathcal{U}}}$ be the contractive projection from ${\mathcal B}(L^2({\Bbb G}))$ onto  $\widetilde{\mathcal H}_\omega=L^\infty(\widehat{\mathbb{G}})$ as defined in Lemma \ref{th0}, which is also contained in $\mathcal{CB}_{\mathcal{T}_\triangleright}({\mathcal B}\left(L^{2}(\mathbb{G})\right))$. Then $G$ is amenable by \cite[Theorem 4.2]{cran} .	
\end{proof}
\begin{lemma}\label{lem1}
	Let ${\Bbb G}$ be a locally compact quantum group. Then $$(\mathcal{T}_\triangleright({\mathbb G})\otimes L^\infty(\widehat{\mathbb{G}})_{\perp})^\perp={\mathcal B}(L^{2}(\mathbb{G})) \bar{\otimes} L^\infty(\widehat{\mathbb{G}}).$$
\end{lemma}
\begin{proof}
	It is clear from the definition	that
	$$
	{\mathcal B}(L^{2}(\mathbb{G})) \bar{\otimes} L^\infty(\widehat{\mathbb{G}})\subseteq(\mathcal{T}_\triangleright({\mathbb G})\otimes L^\infty(\widehat{\mathbb{G}})_{\perp})^\perp
	$$	
	Now, given $u\in (\mathcal{T}_\triangleright({\mathbb G})\otimes L^\infty(\widehat{\mathbb{G}})_{\perp})^\perp$, $\omega_1,\omega_2\in\mathcal{T}_\triangleright({\mathbb G})$ and $\gamma\in L^\infty(\widehat{\mathbb{G}})_{\perp}$, we have 
	$$
	\langle (\omega_1\otimes \iota)(u), \gamma\rangle=(\omega_1\otimes \gamma)(u)=0
	$$
	This shows that $(\omega_1\otimes \iota)(u)\in (L^\infty(\widehat{\mathbb{G}})_{\perp})^\perp=L^\infty(\widehat{\mathbb{G}})$. Obviously,  $(\iota\otimes \omega_2)(u)\in {\mathcal B}(L^{2}(\mathbb{G}))$. Since $\omega_1, \omega_2$ are arbitrary, it follows from \cite[Proposition 2.1]{tomy} that $u\in {\mathcal B}(L^{2}(\mathbb{G})) \bar{\otimes} L^\infty(\widehat{\mathbb{G}})$.
\end{proof}
\begin{theorem}\label{thm01}
	Let ${\Bbb G}$ be a locally compact quantum group. Then $$\langle \mathcal{T}_\triangleright({\mathbb G})\triangleright L^\infty(\widehat{\mathbb{G}})_{\perp}\rangle=L^\infty(\widehat{\mathbb{G}})_{\perp}.$$
\end{theorem}
\begin{proof}
	Suppose that $x\in ({\mathcal{T}_\triangleright({\mathbb G})\triangleright L^\infty(\widehat{\mathbb{G}})_{\perp}})^\perp$. Then 
	$$\langle \widetilde{\Gamma}(x), \omega\otimes \gamma\rangle=\langle x, \omega\triangleright\gamma\rangle=0$$ for all 
	$\omega\in \mathcal{T}_\triangleright({\mathbb G})$ and $\gamma\in L^\infty(\widehat{\mathbb{G}})_{\perp}$. This shows that 
	$\widetilde{\Gamma}(x)\in(\mathcal{T}_\triangleright({\mathbb G})\otimes L^\infty(\widehat{\mathbb{G}})_{\perp})^\perp$ and so 
	$\widetilde{\Gamma}(x)\in{\mathcal B}(L^{2}(\mathbb{G})) \bar{\otimes} L^\infty(\widehat{\mathbb{G}})$ by Lemma \ref{lem1}. On the other hand, it follows from the definition of the map $\widetilde{\Gamma}$	that $\widetilde{\Gamma}(x)\in{\mathcal B}(L^{2}(\mathbb{G})) \bar{\otimes} L^\infty({\mathbb{G}})$. Therefore,  $$(\omega\otimes \iota)(\widetilde{\Gamma}(x))\in L^\infty(\widehat{\mathbb{G}})\cap L^\infty({\mathbb{G}})={\Bbb C}1$$ for all $\omega\in \mathcal{T}_\triangleright({\mathbb G})$. Now, Tomiyama's slice map theorem \cite[Proposition 2.1]{tomy} implies that $$\widetilde{\Gamma}(x)\in{\mathcal B}(L^{2}(\mathbb{G})) \bar{\otimes}{\Bbb C}1={\mathcal B}(L^{2}(\mathbb{G})) {\otimes}1.$$ By Lemma \ref{lem2}, we conclude that $x\in L^\infty(\widehat{\mathbb{G}})$. This shows  that $({\mathcal{T}_\triangleright({\mathbb G})\triangleright L^\infty(\widehat{\mathbb{G}})_{\perp}})^\perp\subseteq L^\infty(\widehat{\mathbb{G}})$.  The other inclusion is trivial and consequently $$\langle \mathcal{T}_\triangleright({\mathbb G})\triangleright L^\infty(\widehat{\mathbb{G}})_{\perp}\rangle=(({\mathcal{T}_\triangleright({\mathbb G})\triangleright L^\infty(\widehat{\mathbb{G}})_{\perp}})^\perp)_\perp=L^\infty(\widehat{\mathbb{G}})_{\perp},$$
	as required.
\end{proof}
Before we can state and prove the main result of this section, we need the following lemma,  whose proof is similar  to those given in  \cite[Lemma 1.1]{wil}. Thus, we omit  the proof here. See also  the proof of 
\cite[Proposition 3.3]{nss}  in the quantum group case.
\begin{lemma}\label{xxx}
	Let ${\Bbb G}$ be a locally compact quantum group  such that $\mathcal{T}_\triangleright({\mathbb G})$ is separable and  let $J$ be a closed
	subspace of $\mathcal{T}_\triangleright({\mathbb G})$ satisfying
	
	\emph{(i)} ${ J}_\omega\subseteq J$ for all $\omega\in S(\mathcal{T}_\triangleright({\mathbb G}))$; and
	
	\emph{(ii)} for every  finite subset $A\subset J$ and every $\varepsilon>0$  there is
	$\omega\in S(\mathcal{T}_\triangleright({\mathbb G}))$ such that $$d(\gamma, { J}_\omega)=\inf\{\|\mu-\gamma\|:~\mu\in{J}_\omega \}<\varepsilon\quad(\gamma\in A).$$
	Then there is $\omega\in S(\mathcal{T}_\triangleright({\mathbb G}))$ such that $J={J}_\omega$.
\end{lemma}

%We have the following lemma  whose proof is similar  to those given in part (a) of \cite[Theorem 2.1]{wil} for locally compact groups. 
%Thus, we omit  the proof.
%\begin{lemma}\label{ll}
%	Let ${\Bbb G}$ be a locally compact quantum group such that $\mathcal{T}_\triangleright({\mathbb G})$ is separable and let $\mathcal{J}=\{J_\omega: \omega\in S(\mathcal{T}_\triangleright({\mathbb G}))\}$. Then every ideal in $\mathcal{J}$ is contained in a maximal one.
%\end{lemma}
The proof of the following result is similar to that of part (a) of \cite[Thorem 1.2]{wil}.
\begin{proposition}\label{p0}
	Let ${\Bbb G}$ be a locally compact quantum group such that $\mathcal{T}_\triangleright({\mathbb G})$ is separable. Then every ideal in the set $\mathcal{J}=\{J_\omega: \omega\in S(\mathcal{T}_\triangleright({\mathbb G}))\}$ is contained in a maximal one.
\end{proposition}	
\begin{theorem}
	Let ${\Bbb G}$ be a locally compact quantum group such that $\mathcal{T}_\triangleright({\mathbb G})$ is separable. Consider the following statements.
	
	\rm{(i)} $\mathbb{G}$ is co-amenable.
	
	\rm{(ii)} $\mathbb{G}$ is amenable.
	
	\rm{(iii)} For every $\omega\in S(\mathcal{T}_\triangleright({\mathbb G}))$, the left ideal $J_\omega$ has a bounded right approximate identity.
	
	\rm{(iv)} The set $\mathcal{J}=\{J_\omega: \omega\in S(\mathcal{T}_\triangleright({\mathbb G}))\}$ has a unique maximal ideal.
	
	Then the following hold. (iv)$\Longleftrightarrow$(ii), (i)$\Longrightarrow$(iii) and (i)$+$(ii)$\Longleftrightarrow$(iii)$+$(iv). Moreover, if (iv) holds, then $L^\infty(\widehat{\mathbb{G}})_{\perp}$ is the unique maximal ideal in $\mathcal{J}$ .
\end{theorem}
\begin{proof}
	(iv)$\Longrightarrow$(ii). Suppose that $\mathcal{J}$ has a unique maximal ideal, say $J_{\omega_0}\in{\mathcal J}$. Since by Proposition \ref{p0}  for each $\omega\in S(\mathcal{T}_\triangleright({\mathbb G}))$ the ideal $J_\omega$ is contained in a maximal ideal in ${\mathcal J}$, it follows that $J_\omega\subseteq J_{\omega_0}$. This means that
	$$J_{\omega_0}^\perp=\{x\in {\mathcal B}(L^2({\Bbb G})): L_\omega(x)=x ~\hbox{for all}~ \omega\in S(\mathcal{T}_\triangleright({\mathbb G})) \}.$$
	Moreover, it is easy to see that $L^\infty(\widehat{\mathbb{G}})\subseteq J_{\omega_0}^\perp$. To prove the converse inclusion, suppose that $x\in J_{\omega_0}^\perp$. Then $L_\omega(x)=\langle \omega, 1\rangle x$ for all $\omega\in \mathcal{T}_\triangleright({\mathbb G})$. Therefore, for every $\omega, \gamma\in \mathcal{T}_\triangleright({\mathbb G})$, we have 
	$$
	\langle \widetilde{\Gamma}(x), \gamma\otimes \omega\rangle=\langle x, \gamma\triangleright \omega\rangle=\langle L_\omega(x), \gamma\rangle=\langle \omega, 1\rangle\langle x, \gamma\rangle=\langle x\otimes 1, \gamma\otimes \omega\rangle. 
	$$
	This shows that $\widetilde{\Gamma}(x)=x\otimes 1$, and so 
	$x\in L^\infty(\widehat{\mathbb{G}})$ by Lemma \ref{lem2}.  Therefore, $\widetilde{\mathcal H}_{\omega_0}=J_{\omega_0}^\perp=L^\infty(\widehat{\mathbb{G}})$, or equivalently, $J_{\omega_0}=L^\infty(\widehat{\mathbb{G}})_\perp$. Now, the result follows from Proposition \ref{ame}.
	
	(i)$\Longrightarrow$(iii). suppose that  $\mathbb{G}$ is co-amenable. Then by \cite[Proposition 5.4]{hnr2011} $\mathcal{T}_\triangleright({\mathbb G})$ has a bounded right approximate identity and so, as described above, for every $\omega\in S(\mathcal{T}_\triangleright({\mathbb G}))$, the left ideal $J_\omega$ has a bounded right approximate identity.
	
	(ii)$\Longrightarrow$(iv).
	Suppose that ${\Bbb G}$ is amenable. Then by a standard argument we can find a net of normal states  $(f_i)$ in $L^1({\Bbb G})$  such that $$\|f\star f_i-\langle f, 1\rangle f_i\|_1\rightarrow 0$$ for all $f\in L^1({\Bbb G})$.  For each $i$, we let $\omega_i\in {\mathcal T}_\triangleright({\Bbb G})$ be a norm preserving  normal extension  of $f_{i}$. Given $\gamma\in L^\infty(\widehat{\Bbb G})_\perp$, $\sigma\in \mathcal{T}_\triangleright({\mathbb G})$ and $x\in {\mathcal B}(L^2({\Bbb G}))$, let $f=\pi(\gamma)\in L^1({\Bbb G})$. Then since $\langle f, 1\rangle=0$ and   $x\triangleright \sigma\in L^\infty({\Bbb G})$, we conclude that 
	\begin{eqnarray*}
		\langle(\sigma\triangleright\gamma)\triangleright\omega_{i}, x\rangle
		=\langle\gamma\triangleright\omega_{i}, x\triangleright\sigma\rangle
		=\langle f\star f_{i}, x\triangleright\sigma\rangle\rightarrow 0.
	\end{eqnarray*}
	As $\langle \mathcal{T}_\triangleright({\mathbb G})\triangleright L^\infty(\widehat{\mathbb{G}})_{\perp}\rangle=L^\infty(\widehat{\mathbb{G}})_{\perp}$, by Theorem \ref{thm01}, it follows that  $$\langle \gamma\triangleright\omega_{i}, x\rangle\rightarrow 0$$ for all $x\in {\mathcal B}(L^2({\Bbb G}))$ and $\gamma\in L^\infty(\widehat{\Bbb G})_\perp$. Applying Mazur's theorem, we can obtain a net  of normal states $(\omega_{i})$ in $\mathcal{T}_\triangleright({\mathbb G})$   such that $\|\gamma\triangleright\omega_{i}\|\rightarrow 0$ for all $\gamma\in L^\infty(\widehat{\Bbb G})_\perp$. This shows that the conditions of
	Lemma \ref{xxx} are satisfied if we take $J=L^\infty(\widehat{\Bbb G})_\perp$ and hence there is a state $\omega$ in  $\mathcal{T}_\triangleright({\mathbb G})$ such that $J_\omega=L^\infty(\widehat{\Bbb G})_\perp$.
	Now, the equivalence  (iii)$+$(iv)$\Longleftrightarrow$(i)$+$(ii) follows from Theorem \ref{a-co}.
\end{proof}
\section{{The discrete quantum group case}}\label{sec5}
In this section, we consider discrete quantum groups ${\mathbb G}$. Since ${\mathbb G}$ is discrete, it follows from \cite[Theorem 3.7]{hnr} that $ {\mathcal K}(L^2({\Bbb G}))$ is a Banach $\mathcal{T}_\triangleright({\mathbb G})$-submodule of ${\mathcal B}\left(L^{2}(\mathbb{G})\right)$ when we consider the canonical
$\mathcal{T}_\triangleright({\mathbb G})$-bimodule  structure on
${\mathcal B}\left(L^{2}(\mathbb{G})\right)$. It is also known that the multiplier algebra of the $C^*$-algebra ${\mathcal K}(H)$ of compact operators on the Hilbert space $H$, is equal to ${\mathcal B}(H)$. 

\begin{lemma}\label{decom}
	Let ${\Bbb G}$ be a discrete quantum group. Then the algebra ${\mathcal B}(L^2({\Bbb G}))^*$, equipped with the left Arens product, can be decomposed as
	$$
	{\mathcal B}(L^2({\Bbb G}))^*=\mathcal{T}_\triangleright({\mathbb G})\oplus_1 {\mathcal K}(L^2({\Bbb G}))^\perp,
	$$ 
	of the closed subalgebra $\mathcal{T}_\triangleright({\mathbb G})$  and the weak$^*$ closed ideal ${\mathcal K}(L^2({\Bbb G}))^\perp$.
\end{lemma}
\begin{proof}
	Clearly, ${\mathcal K}(L^2({\Bbb G}))^\perp$ is a weak$^*$ closed subspace of ${\mathcal B}(L^2({\Bbb G}))^*$ and by \cite[Proposition 1.5]{nss} we have the Banach space decomposition 
	$\mathcal{T}_\triangleright({\mathbb G})\oplus_1 {\mathcal K}(L^2({\Bbb G}))^\perp$. Thus it  suffice to prove that ${\mathcal K}(L^2({\Bbb G}))^\perp$ is an ideal in ${\mathcal B}(L^2({\Bbb G}))^*$. To prove this,  
	fix $n\in {\mathcal K}(L^2({\Bbb G}))^\perp$ and  $m\in {\mathcal B}(L^2({\Bbb G}))^*$, and $x\in {\mathcal K}(L^2({\Bbb G}))$. Then $n\square x=0$  since  $x\triangleright\gamma\in {\mathcal K}(L^2({\Bbb G}))$; see \cite[Theorems 3.1 and 3.7]{hnr} and hence $$\langle n\square x, \gamma\rangle=\langle n, x\triangleright\gamma\rangle=0$$ for all $\gamma \in \mathcal{T}_\triangleright({\mathbb G})$.  This implies that ${\mathcal K}(L^2({\Bbb G}))^\perp$ is a left ideal in ${\mathcal B}(L^2({\Bbb G}))^*$. To prove that ${\mathcal K}(L^2({\Bbb G}))^\perp$ is a right ideal in ${\mathcal B}(L^2({\Bbb G}))^*$, put  $\gamma_0=m|_{{\mathcal K}(L^2({\Bbb G}))}\in \mathcal{T}_\triangleright({\mathbb G})$. Then it is easy to see that  $m\square x=\gamma_0\triangleright x\in {\mathcal K}(L^2({\Bbb G}))$. This shows that $n\square m\in {\mathcal K}(L^2({\Bbb G}))^\perp$, as required.
\end{proof}
\begin{lemma}\label{ip}
	Let ${\Bbb G}$ be a discrete quantum group and let  $m\in {\mathcal B}(L^2({\Bbb G}))^*$ be a contractive idempotent. Then either 
	$m\in\mathcal{T}_\triangleright({\mathbb G})$  or $m\in {\mathcal K}(L^2({\Bbb G}))^\perp$.
\end{lemma}
\begin{proof}
	Suppose that $m=\omega+\gamma$, where $\omega\in\mathcal{T}_\triangleright({\mathbb G})$  and $\gamma\in {\mathcal K}(L^2({\Bbb G}))^\perp$.  By assumption and this fact that ${\mathcal K}(L^2({\Bbb G}))^\perp$ is an ideal in ${\mathcal B}(L^2({\Bbb G}))^*$, we obtain $\omega\triangleright\omega=\omega$. Moreover, by Lemma \ref{decom}, we have $\|m\|=\|\omega\|+\|\gamma\|$, which implies that $\omega$ is a contractive idempotent. Therefore, either $\|\omega\|=1$ or $\|\omega\|=0$. If $\|\omega\|=1$, then $\|\gamma\|=0$ and hence 
	$m=\omega\in \mathcal{T}_\triangleright({\mathbb G})$. If $\|\omega\|=0$, then $m=\gamma\in {\mathcal K}(L^2({\Bbb G}))^\perp$. 
\end{proof}

\begin{theorem}\label{ff}
	Let ${\Bbb G}$ be a discrete quantum group and let  $\omega\in \mathcal{T}_\triangleright({\mathbb G})$ with $\|\omega\|=1$. 
	Then the following statements are equivalent.
	
	{\rm (i)} The Cesaro sums $\omega_n$ do not converge to $0$ in the 
	weak$^*$ topology of  $\mathcal{T}_\triangleright({\mathbb G})$.
	
	{\rm (ii)} The weak$^*$ limit $\widetilde{\omega}=w^*-\lim_{n\rightarrow \infty}\omega_n$
	exists and $\widetilde{\omega}$ is a non-zero contractive idempotent in $\mathcal{T}_\triangleright({\mathbb G})$.
	
	{\rm (iii)} There is a  free ultrafilter ${\mathcal U}$ on ${\Bbb N}$ such that the functional  $m_{\mathcal U}=w^*-\lim_{\mathcal U}\omega_n$ is a non-zero contractive idempotent in $\mathcal{T}_\triangleright({\mathbb G})$ 
	
	{\rm (iv)} $\widetilde{\mathcal H}_\omega\cap {\mathcal K}(L^2({\Bbb G}))\neq\{0\}$. 
	
	{\rm (v)} There is $\gamma\in \mathcal{T}_\triangleright({\mathbb G})\setminus L^\infty({\Bbb G})_\perp$ such that $\omega\triangleright\gamma=\gamma$.
\end{theorem}
\begin{proof}
	(i)$\Rightarrow$(ii). Suppose that the Cesaro sums  $\omega_n$ do not converge to $0$ in the 
	weak$^*$ topology of  $\mathcal{T}_\triangleright({\mathbb G})$. Then there is a  free ultrafilter ${\mathcal U}$ on ${\Bbb N}$ such that $m_{\mathcal U}=w^*-\lim_{\mathcal U}\omega_n$ is non-zero on ${\mathcal K}(L^2({\Bbb G}))$. Thus $m_{\mathcal U}$ is a contractive idempotent in $\mathcal{T}_\triangleright({\mathbb G})$ by Lemma \ref{ip}.  Since the sequence $(\omega_n)_{n\in {\Bbb N}}$
	is bounded, every subnet of it  has a subnet converging weak$^*$ to some $m_{\mathcal V}$ with respect to some free ultrafilter ${\mathcal V}$ on ${\Bbb N}$. Since $\omega\triangleright m_{\mathcal U}=m_{\mathcal U}$, we obtain that $\omega_n\triangleright m_{\mathcal U}=m_{\mathcal U}$ for all $n\in{\Bbb N}$. This shows that 
	$$m_{\mathcal V}\square m_{\mathcal U}=m_{\mathcal V}\triangleright m_{\mathcal U}=w^*-\lim_{\mathcal V}(\omega_n\triangleright m_{\mathcal U})=m_{\mathcal U}.$$
	Again by Lemma \ref{ip} we give that $m_{\mathcal V}$ is also a non-zero contractive idempotent in $\mathcal{T}_\triangleright({\mathbb G})$. Moreover, by \cite[Theorem 3.7]{hnr} discretness of ${\Bbb G}$ implies that  the convolution $\triangleright$ on $\mathcal{T}_\triangleright({\mathbb G})$ is weak$^*$ continuous on the right. Therefore, $$m_{\mathcal U}=m_{\mathcal V}\triangleright m_{\mathcal U}=w^*-\lim_{\mathcal U}(m_{\mathcal V}\triangleright \omega_n)=m_{\mathcal V}.$$
	This shows that $m_{\mathcal U}$ is the only weak$^*$
	cluster point of the sequence $(\omega_n)_{n\in {\Bbb N}}$. 
	Thus, the sequence  $(\omega_n)_{n\in {\Bbb N}}$ converges weak$^*$ in $\mathcal{T}_\triangleright({\mathbb G})$ to $m_{\mathcal U}$.
	
	The implication	(ii)$\Rightarrow$(iii) is trivial.	
	
	(iii)$\Rightarrow$(iv). Since $m_{\mathcal U}$ is a non-zero idempotent, we can find $x\in {\mathcal K}(L^2({\Bbb G}))$ such that $y:=m_{\mathcal U}\triangleright x\neq 0$. Moreover, it is clear that $y\in \widetilde{\mathcal H}_\omega\cap {\mathcal K}(L^2({\Bbb G}))$.

	(iv)$\Rightarrow$(i). Let $x\in\widetilde{\mathcal H}_\omega\cap {\mathcal K}(L^2({\Bbb G}))$ be non-zero. Then there is 
	$\gamma\in \mathcal{T}_\triangleright({\mathbb G})$ such that $\langle x, \gamma\rangle\neq 0$. This implies that $\langle \omega_n, x\triangleright \gamma\rangle=\langle \omega_n\triangleright x, \gamma\rangle=\langle x, \gamma\rangle$ for all $n\in {\Bbb N}$.
	This shows that the Cesaro sums $\omega_n$ do not converge to $0$ in the 
	weak$^*$ topology of  $\mathcal{T}_\triangleright({\mathbb G})$.

	(iii)$\Rightarrow$(v). It suffice to take $\gamma=m_{\mathcal U}$.

	(v)$\Rightarrow$(iv). Since $\gamma\in \mathcal{T}_\triangleright({\mathbb G})\setminus L^\infty({\Bbb G})_\perp$, there is $x\in {\mathcal K}(L^2({\Bbb G}))$ such that $y:=L_\gamma(x)=\gamma\triangleright x\neq 0$. On the other hand, by assumption, we have
	$$
	L_\omega(y)=L_\omega(L_\gamma(x))=L_{\omega\triangleright\gamma}(x)=L_\gamma(x)=y, 
	$$
	which implies that $y\in \widetilde{\mathcal H}_\omega\cap {\mathcal K}(L^2({\Bbb G}))$.
\end{proof}

\begin{corollary}
	Let ${\Bbb G}$ be a discrete and infinite quantum group and let  $\omega\in S(\mathcal{T}_\triangleright({\mathbb G}))$ be non-degenerate. Then  the Cesaro sums $\omega_n$ converge to $0$ in the 
	weak$^*$ topology of  $\mathcal{T}_\triangleright({\mathbb G})$.  
\end{corollary}
\begin{proof}
This is an immediate consequence of Theorems \ref{com} and \ref{ff}.
\end{proof}

\end{document}